[1994/12/01]

\documentclass[a4paper, 12pt]{amsart}
\usepackage{amssymb,amsxtra}
\usepackage{amsmath,amsthm,amscd}
\usepackage[all,cmtip]{xy}

\chardef\bslash=`\\ 

\newtheorem{thm}{Theorem}[section]
\newtheorem{cor}[thm]{Corollary}
\newtheorem{lem}[thm]{Lemma}
\newtheorem{prop}[thm]{Proposition}

\theoremstyle{definition}
\newtheorem{defn}{Definition}[section]

\theoremstyle{remark}
\newtheorem{rem}{Remark}[section]

\newcommand{\thmref}[1]{Theorem~\ref{#1}}

\newcommand{\lemref}[1]{Lemma~\ref{#1}}

\newcommand{\defnref}[1]{Definition~\ref{#1}}

\newcommand{\remref}[1]{Remark~\ref{#1}}

\newcommand{\eval}[2][\right]{\relax
  \ifx#1\right\relax \left.\fi#2#1\rvert}

\def\R{{\mathbf{R}}}
\def\C{{\mathbf{C}}}

\def\t{{\mathbf{T}}}

\newcommand\go{G^{(0)}}

\title{The C*-algebra of a twisted groupoid extension}
\author{Jean N. Renault}
\address{Institut Denis Poisson (UMR 7013)\\
  Universit\'e d'Orl\'eans et CNRS \\ 45067
  Orl\'eans Cedex 2, FRANCE}
\email{jean.renault@univ-orleans.fr}
\keywords{groupoid extension, twist, Mackey machine, multiplier representation.}

\usepackage{pdfsync}
\begin{document}
\vskip5mm
\begin{abstract} This article extends the main results of \cite{ikrsw:extensions} to the case of a twisted groupoid. More precisely, it  gives a decomposition of the C*-algebra of a twisted locally compact groupoid with Haar system in presence of a normal subgroupoid. When the normal subgroupoid and the twist over it are abelian, one obtains another twisted groupoid C*-algebra. This is applied to C*-algebraic deformation quantization and to multiplier representations of locally compact abelian groups.
\end{abstract}
\maketitle
\markboth{Jean Renault}
{Twisted groupoid extensions}

\renewcommand{\sectionmark}[1]{}

\section*{Introduction.}

The Mackey normal subgroup analysis (also called the Mackey machine) describes the representations of a group $G$ in terms of a normal subgroup $S$ and of the quotient $H=G/S$. 
A semidirect product of groups 
$G=S\rtimes H$
such as the group of rigid motions or the Poincar\'e group is the simplest example. From the C*-algebraic perspective, it gives a description of $C^*(G)$ as a crossed product. In the simple case of a semidirect product, we easily have
$$C^*(S\rtimes  H)=C^*(S)\rtimes H=C^*(H, C^*(S))$$

A semidirect product is a trivial extension. When the extension is not trivial, a twist appears. We use here the formalism of twisted crossed products as defined by P.~Green in \cite{gre:local}.
\begin{thm}\cite[Proposition 1] {gre:local}\label{Green} Let $S$ be a closed normal subgroup of a locally compact group $G$. Then
$$C^*(G)=C^*(G, C^*(S),\tau_S)$$
 where the right handside is a twisted crossed product.
 \end{thm}
When $S$ is abelian, one can go one step further, namely using the Gelfand transform, one has
 $$C^*(S\rtimes  H)=C^*(H, C^*(S))=C^*(H, C_0(\hat S))=C^*(\hat S\rtimes H)$$
The last term is a groupoid C*-algebra, where the groupoid $\hat S\rtimes H$ has less isotropy than the initial group $S\rtimes  H$. One may want to iterate the process. It is then necessary to extend the Mackey machine to a groupoid $G$ rather than to a group. The original motivation of this work was the analysis of nilpotent group C*-algebras. However, the examples and the applications, presented in the last section, are limited here to the Heisenberg group and to twisted locally compact abelian groups. This article is based on a joint work \cite{ikrsw:extensions} with M.~Ionescu, A.~Kumjian, A.~Sims and D.~ Williams which considered the case of an untwisted groupoid. The proofs for twisted groupoids are essentially the same; moreover, the twisted case can be deduced directly from the untwisted case, as in \cite[Proposition 3.5]{ikrsw:pushouts}. Nevertheless, it seems preferable to give a complete presentation of the general result for further references and for applications. The main results give isomorphisms of C*-algebras. It is equally important to understand these results as correspondences between representations of groupoids or groupoid dynamical systems. This is done here, using  the disintegration theorem of \cite{ren:rep} as the essential tool. This theorem is recalled in the first section and most of the paper relies on \cite{ren:rep} only. The main theorem \ref{main} is illustrated by two applications, given in the last section. First, it shows that two C*-algebraic deformation quantization of a symplectic linear space, namely Connes' deformation based on the tangent groupoid and Rieffel's strict deformation quantization, are the same. Second, the main results of \cite{bk:multiplier} on multiplier representations of locally compact abelian groups are easily recovered, in a fashion very close to \cite[Section 7]{gre:local}, but with the advantage of giving explicit pictures of the C*-algebra of the twisted abelian group.\\

 In order to use the disintegration theorem \cite[4.1]{ren:rep}, we shall usually assume that the locally compact groupoids are second countable and that the C*-algebras and C*-bundles appearing in the definition of a groupoid C*-dynamical system are separable. Note however that, in the framework of groups and twisted group crossed products studied in \cite{gre:local}, one can write a representation as an integrated representation without this assumption. The following notation will be frequently used: given two maps $p:X\to T$ and $q:Y\to T$ with the same range, their fibered product over $T$ is denoted by $X*Y$ when there is no ambiguity about the maps. To distinguish bundles of algebras (or of linear spaces) from algebras, algebra bundles will usually be denoted by calligraphic letters such as $\mathcal A$ while algebras will be denoted by Roman letters such as $A$. For example, if $p:G\twoheadrightarrow X$ is a bundle of groupoids with Haar systems, ${\mathcal C}_c(G)$ and ${\mathcal C}^*(G)$ denote the bundles with respective fibers $C_c(G(x))$ and $C^*(G(x))$.  On the other hand, $C_c(G)$ and $C^*(G)$ denote the usual $*$-algebras of the groupoid $G$, which are the sectional algebras of the above bundle.

\section{Groupoid C*-dynamical systems and their C*-algebras}
\begin{defn}
A groupoid extension is a short exact sequence of groupoids
\[S\rightarrowtail G \twoheadrightarrow H\]
with common unit space $G^{(0)}$. Equivalently, an extension of the groupoid $H$ is a surjective homomorphism $\pi:G\twoheadrightarrow H$ such that $\pi^{(0)}:G^{(0)}\to H^{(0)}$ is a bijection (we shall assume that $G^{(0)}=H^{(0)}$ and that $\pi^{(0)}$ is the identity map). We shall write $\dot\gamma=\pi(\gamma)$ when there is no ambiguity about the projection map.
\end{defn}
Then $S=\text{Ker}(\pi)$ is a subgroup bundle of the isotropy group bundle $G'$. Moreover, it is normal in the sense that for all compatible pair $(\gamma,s)\in G*S$, $\gamma s \gamma^{-1}$ belongs to $S$. Note that subgroupoids which are normal in this sense are necessarily subgroup bundles of the isotropy group bundle. Note also that $H$ is naturally isomorphic to the quotient groupoid $G/S$. Since the normal subgroupoid $S$ of $G$ determines the extension, we shall often use the terminology of normal subgroupoid rather than extension. Then $H$ is the quotient groupoid.  We assume that $H$ and $G$ are locally compact Hausdorff groupoids and that $\pi$ is continuous and open. In particular $S$ is a closed normal subgroupoid of $G$. We also assume that $H$ has a Haar system $\alpha$ (here and in the rest of the paper, Haar systems are assumed to be continuous, as in \cite[Definition I.2.2]{ren:approach}) and that $S$ has a Haar system $\beta$. There is a homomorphism $\delta: G\to\R_+^*$ such that for all $\gamma\in G$, $\gamma\beta^{s(\gamma)}\gamma^{-1}=\delta(\gamma)\beta^{r(\gamma)}$. This homomorphism is called the modular cocycle of the extension. Its cohomology class does not depend on the choice of $\beta$. For all $x\in G^{(0)}$, its restriction to the group $S_x$ is the modular function of  $S_x$. It is continuous (see \cite[Lemma 2.4]{ikrsw:extensions}). Given $\alpha$ and $\beta$, we define the Haar system $\lambda$ for $G$ by the formula:
\[\int f(\gamma) d\lambda^x(\gamma)=\int_H\int_S f(\gamma t)d\beta^{s(\gamma)}(t)d\alpha^x(\dot\gamma)\]

\begin{defn}\label{extension}
 Under these assumptions, we say that $S\rightarrowtail G \twoheadrightarrow H$ is a locally compact groupoid extension with Haar systems.
 \end{defn}
 
 It will be convenient in the sequel to define an extension with Haar system as a pair $(G,S)$ where $S$ is a closed normal subgroupoid of the locally compact groupoid $G$ which admits a Haar system and such that the quotient groupoid $H=G/S$ admits a Haar system. We shall keep this convention and this notation throughout the paper.

\begin{defn}\label{dynamical system}\cite{ren:rep}
 A groupoid C*-dynamical system (or dynamical system for short) is a triple $(G, S, {\mathcal A})$ where $G$ is a locally compact groupoid, $S$ is a closed normal subgroupoid, and $\mathcal A$ is an upper semi-continuous bundle of C*-algebras over $G^{(0)}$ endowed with a continuous action $G*{\mathcal A}\to {\mathcal A}$ such that  $S$ is unitarily implemented in the multiplier algebra bundle $M({\mathcal A})$, meaning the existence of a bundle homomorphism $\chi$ from $S$ to the unitary bundle of $M({\mathcal A})$, such that
\begin{enumerate}
 \item the map $S*{\mathcal A}\to {\mathcal A}$ sending $(s,a)$ to $\chi(s)a$ is continuous;
 \item $s.a = \chi(s)a\chi(s)^{-1}$ for all $(s,a) \in S*{\mathcal A}$;
 \item $\chi(\gamma s\gamma^{-1}) = \gamma.\chi(s)$ for all  $(\gamma,s) \in
G*S$.
\end{enumerate}
\end{defn}
In \cite[Section 3]{ren:rep}, it was assumed that the kernel $S$ was abelian, in the sense that it was a bundle of abelian groups. However, as shown in \cite{ikrsw:extensions}, this assumption is not necessary and most results of \cite{ren:rep} remain valid.\\

Recall the construction of the crossed products, as in \cite[Section 3]{ren:rep}. Let $(G, S, {\mathcal A})$ be a groupoid C*-dynamical system. We assume that $H=G/S$ has a Haar system $(\alpha^x)_{x\in G^{(0)}}$ and that $S$ has a Haar system $(\beta^x)_{x\in G^{(0)}}$. One first form the $*$-algebra $C_c(G, S, {\mathcal A})$. Its elements are continuous functions $f:G\to {\mathcal A}$ such that
\begin{enumerate}
 \item $f(\gamma)$ belongs to $A_{r(\gamma)}$ for all $\gamma\in G$;
 \item $f(s\gamma)=f(\gamma)\chi(s^{-1})$ for all $(s,\gamma)\in S*G$;
 \item $f$ has compact support modulo $S$.
\end{enumerate}
The product and the involution are respectively given by
\[f*g(\gamma)=\int f(\tau)[\tau.g(\tau^{-1}\gamma)] d\alpha^{r(\gamma)}(\dot\tau)\]
and
\[f^*(\gamma)=\gamma.(f(\gamma^{-1}))^*\]
One defines the I-norm as
\[\|f\|_I=\max\big(\sup_x\int \|f(\gamma)\|d\alpha^x(\dot\gamma),\sup_x\int \|f(\gamma^{-1})\|d\alpha^x(\dot\gamma)\big).\]
The crossed product C*-algebra $C^*(G,S,{\mathcal A})$ is the completion of $C_c(G,S,{\mathcal A})$ for the full C*-norm which is defined as $\|f\|=\sup_L \|L(f)\|$ where $L$ runs into the set of all representations of $C_c(G,S,{\mathcal A})$ in Hilbert spaces which are I-norm decreasing. 

\begin{defn}\label{representation}\cite[D\'efinition 3.4]{ren:rep} A representation of a dynamical system $(G,S,{\mathcal A})$ is a pair $(\mu,{\mathcal H})$ where $\mu$ is a measure on $G^{(0)}$, quasi-invariant with respect to $(H,\alpha)$ with module $\Delta$ and ${\mathcal H}$ is a measurable Hilbert bundle over $G^{(0)}$, endowed with a unitary representation of $G$ and a representation of ${\mathcal A}$ such that
\begin{enumerate}
 \item $\gamma(a\xi)=(\gamma a)(\gamma\xi)$ for all $(\gamma, a,\xi)\in G*{\mathcal A}*{\mathcal H}$;
 \item $s\xi=\chi(s)\xi$ for all $(s,\xi)\in S*{\mathcal H}$.
\end{enumerate}
It is called non-degenerate if the representation of ${\mathcal A}$ is non-degenerate.
 \end{defn}
 
 In \cite{ren:rep}, it is assumed that ${\mathcal H}$ is defined over a Borel subset $U\subset G^{(0)}$ whose saturation $[U]$ is $\mu$-conull. However, as shown for example in \cite[Section 8]{wil:tk}, there is no loss of generality in assuming that ${\mathcal H}$ is defined over  $G^{(0)}$.

\begin{defn}\label{integrated} The integrated representation of the representation $(\mu,{\mathcal H})$ of the dynamical system $(G,S,{\mathcal A})$ is the representation $L$ of $C_c(G,S,{\mathcal A})$ on the Hilbert space $L^2(G^{(0)},\mu;{\mathcal H})$ such that for $f\in C_c(G, S, {\mathcal A})$ and $\xi,\eta\in L^2(G^{(0)},\mu;{\mathcal H})$, one has
\[(\xi\,|\,L(f)\eta)=\int (\xi\circ r(\gamma)\,|\, f(\gamma)\,[\gamma(\eta\circ s(\gamma))])\,\Delta^{-1/2}(\dot\gamma)\,d\alpha^x(\dot\gamma)d\mu(x)\]
for a.e. $x$.
\end{defn}

Note that the integrated representation of a non-degenerate representation is non-degenerate. A straightforward application of the Cauchy-Schwarz inequality shows that integrated representations are I-norm decreasing.  The disintegration theorem \cite[Theorem 4.1]{ren:rep} says that, if $G$ is second countable and $\mathcal A$ is separable, every non-degenerate representation $L$ which is  continuous for the inductive limit topology  is unitarily equivalent to an integrated representation and therefore is I-norm decreasing. Since I-norm decreasing representations are necessarily continuous for the inductive limit topology, the three properties: I-norm decreasing, continuous for the inductive limit topology and unitarily equivalent to an integrated representation, coincide in this case for non-degenerate representations of $C_c(G,S,{\mathcal A})$ and define exactly the non-degenerate representations of $C^*(G,S,{\mathcal A})$. The fact that every non-degenerate representation of $C^*(G,S,{\mathcal A})$ is unitarily equivalent to an integrated representation is also true under other assumptions, for example if $G$ is a group.

\section{Mackey analysis of a twisted groupoid C*-algebra.}

\subsection{Twists}

We have given earlier the general notion of an extension. The following special case has been introduced by Kumjian in \cite{kum:diagonals} in the framework of groupoids.

\begin{defn} A central groupoid extension 
$$G^{(0)}\times\t\rightarrowtail\Sigma\twoheadrightarrow G$$ 
where $\t$ is the group of complex numbers of module 1, is called a twist. Then, we say that $(G,\Sigma)$ is a twisted groupoid.
\end{defn} 
We need to distinguish arbitrary extensions as above and twists, because they do not play the same role in this study. While a twisted groupoid is denoted by $(G,\Sigma)$ where $\Sigma$ is the middle term, an arbitrary extension will be determined by a closed normal subgroupoid and denoted for example by $(G,S)$ where $S$ is the kernel of the extension.

\subsection{Twisted extensions}

Let $(G,\Sigma)$ be a twisted groupoid. Then $(\Sigma, G^{(0)}\times\t, G^{(0)}\times\C)$ is a groupoid dynamical system with the action $\sigma(s(\sigma), a)=(r(\sigma),a)$ for $(\sigma,a)\in\Sigma\times\C$ and $\chi(x,\theta)=\theta$ for $(x,\theta)\in G^{(0)}\times\t$. If $G$ is a locally compact groupoid with Haar system, we can construct the crossed product C*-algebra, which we denote by $C^*(G,\Sigma)$ rather than $C^*(\Sigma, G^{(0)}\times\t, G^{(0)}\times\C)$ and which we call the twisted groupoid C*-algebra. The principle of our version of Mackey analysis is to decompose this C*-algebra when $G$ possesses a closed normal subgroupoid $S$ endowed with a Haar system. This is a strong condition, which is often not satisfied by the isotropy bundle itself.

\begin{defn}\label{twisted extension}
 We call $(G,\Sigma, S)$ a twisted extension with Haar systems when $(G,\Sigma)$ is a twisted groupoid, $S$ is a closed normal subgroupoid of $G$ with Haar system and $G/S$ has a Haar system.
\end{defn}
In the sequel, we shall denote $H=G/S$ the quotient groupoid and $\pi: G\to H$ the quotient map.
We shall denote by $(\alpha^x)_{x\in G^{(0)}}$ [resp. $(\beta^x)_{x\in G^{(0)}}$] the Haar system of $H$ [resp. $S$ ]. We denote by
$(\lambda^x)_{x\in G^{(0)}}$ the Haar system of $G$ described earlier. The following diagram summarizes the situation:
\[\xymatrix{
G^{(0)}\times\t\ar@{^{(}->}[d] \ar@{=}[r] &G^{(0)}\times\ar@{^{(}->}[d] \t&\\
\Sigma_{|S}\ar@{->>}[d]\ar@{^{(}->}[r] &\Sigma \ar@{->>}[d]^p \ar@{->>}[r]^{\pi_\Sigma} &H\ar@{=}[d] \\
S\ar@{^{(}->}[r] &G \ar@{->>}[r]^\pi &H\\
}\]

\subsection{The tautological groupoid dynamical system}

It can be shown just as in the untwisted case (see for example \cite[Theorem 5.5]{lr:quantization}) that the twisted group bundle $(S,\Sigma_{|S})$ with Haar system $\beta$  defines an upper semi-continuous bundle of C*-algebras over $G^{(0)}$, with fibers $C^*(S_x,\Sigma_{|S_x})$  which we denote by  ${\mathcal C}^*(S,\Sigma_{|S})$. 

It is endowed with an action  $\Sigma*{\mathcal C}^*(S,\Sigma_{|S})\to {\mathcal C}^*(S,\Sigma_{|S})$, where
$$(\tau.f)(\sigma)=\delta(p(\tau))f(\tau^{-1}\sigma\tau),\quad{\rm for}\quad \tau\in\Sigma,\, f\in C_c(S_{s(\tau)},\Sigma_{|S_{s(\tau)}}), \,\sigma\in p^{-1}(S_{r(\tau)})$$
The introduction of $\delta$ is necessary in order to preserve the convolution product of the twisted group C*-algebras $C^*(S_x,\Sigma_{|S_x})$.
 
\begin{prop} The triple $(\Sigma, \Sigma_{|S}, {\mathcal C}^*(S,\Sigma_{|S}))$ is a groupoid C*-dynamical system, which we call the tautological groupoid dynamical system of the twisted extension. 
\end{prop}
 
 \begin{proof} The continuity of the action map $\Sigma*{\mathcal C}^*(S,\Sigma_{|S})\to {\mathcal C}^*(S,\Sigma_{|S})$ is proved just as \cite[Proposition 2.7]{ikrsw:extensions}. The action of $\Sigma_{|S}$ on ${\mathcal C}^*(S,\Sigma_{|S})$ is implemented by the bundle homomorphism $\chi: \Sigma_{|S}\to UM({\mathcal C}^*(S,\Sigma_{|S})$ which associates to $\sigma\in p^{-1}(S_x)$ the canonical unitary $\chi(\sigma)$ in the multiplier algebra of $C^*(S_x, p^{-1}(S_x))$. Explicitly, for  $f\in C_c(S_x,\Sigma_{|S_x})$ and $\sigma,\tau\in p^{-1}(S_x)$,
 \[(\chi(\sigma)f)(\tau)=\delta^{1/2}(p(\sigma))f(\sigma^{-1}\tau)\]
 One checks just as in \cite[Section 2]{ikrsw:extensions} that the conditions of \defnref{dynamical system} are satisfied. 
 \end{proof}

The relation between the twisted groupoid $(G,\Sigma)$ and the groupoid dynamical system $(\Sigma, \Sigma_{|S}, {\mathcal C}^*(S,\Sigma_{|S}))$ is very clear at the level of their representations:

\begin{prop}\label{representation} Let $(G,\Sigma)$ be a twisted locally compact groupoid with Haar system and let $S\subset G$ be a closed normal subgroupoid with Haar system. Then $(G,\Sigma)$ and $(\Sigma, \Sigma_{|S}, {\mathcal C}^*(S,\Sigma_{|S}))$ have the same representations.
\end{prop}

\begin{proof} As recalled earlier, a representation of $(G,\Sigma)$ is a pair  $(\mu,{\mathcal H})$ where $\mu$ is a measure on $G^{(0)}$ which is quasi-invariant with respect to $(G,\lambda)$ and $\mathcal H$ is a measurable Hilbert bundle over $G^{(0)}$ endowed with a unitary representation $\Sigma*{\mathcal H}\to {\mathcal H}$ such that $(x,t)\xi=t\xi$ for $(x,t)\in G^{(0)}\times\t$ and $\xi\in H_x$. On the other hand, a representation of $(\Sigma, \Sigma_{|S}, {\mathcal C}^*(S,\Sigma_{|S}))$ is a pair $(\underline\mu,\underline{\mathcal H})$ where $\underline\mu$ is a measure on $G^{(0)}$ which is quasi-invariant with respect to $(H,\alpha)$ and $\underline{\mathcal H}$ is a measurable Hilbert bundle over $G^{(0)}$ endowed with a unitary representation $\Sigma*\underline{\mathcal H}\to \underline{\mathcal H}$ and a non-degenerate representation ${\mathcal C}^*(S,\Sigma_{|S})*\underline{\mathcal H}\to \underline{\mathcal H}$ such that
\begin{enumerate}
 \item $\tau(a\xi)=(\tau a)(\tau\xi)$ for all $(\tau, a,\xi)\in G*{\mathcal C}^*(S,\Sigma_{|S})*{\mathcal H}$;
 \item $\sigma\xi=\chi(\sigma)\xi$ for all $(\sigma,\xi)\in (\Sigma_{|S})*{\mathcal H}$.
\end{enumerate}
We show that $(\mu,{\mathcal H})$ can be viewed as a representation of $(\Sigma, \Sigma_{|S}, {\mathcal C}^*(S,\Sigma_{|S}))$ and conversely. According to \cite[Proposition 5.3.10]{adr:amenable}, a measure $\mu$ on $G^{(0)}$ is quasi-invariant with respect to $(G,\lambda)$ if and only if it is quasi-invariant with respect to $(H,\alpha)$. Their respective Radon-Nikodym derivatives $\Delta$ and $\underline\Delta$ are related by $\Delta(\gamma)=\delta(\gamma)\underline\Delta\circ\pi(\gamma)$. Moreover a $(\Sigma, G^{(0)}\times\t)$-Hilbert bundle and a $(\Sigma, {\mathcal C}^*(S,\Sigma_{|S}))$-Hilbert bundle as defined above are in fact the same object. Indeed, let $\mathcal H$ be a $(\Sigma, G^{(0)}\times\t)$-Hilbert bundle. Then by restriction, it is a $(\Sigma_{|S}, G^{(0)}\times\t)$-Hilbert bundle. Just as in the group case, integration gives a one-to-one correspondence between $(\Sigma_{|S}, G^{(0)}\times\t)$-Hilbert bundle and non-degenerate ${\mathcal C}^*(S,\Sigma_{|S})$-Hilbert bundles. The explicit formula for $a\in C_c(S_x,\Sigma_{|S_x})$ and $\xi\in {\mathcal H}_x$ is
\[a\xi=\int a(\sigma)(\sigma\xi)\delta^{-1/2}(\dot\sigma)d\beta^x(\dot\sigma).\]
One checks easily that the relation (i) is satisfied. The relation (ii) simply expresses the fact that the representation of ${\mathcal C}^*(S,\Sigma_{|S})$ on $\mathcal H$ is the integrated form of the restriction to $\Sigma_{|S}$ of the unitary representation of $\Sigma$.
\end{proof}

One deduces that the C*-algebras are isomorphic:

\begin{thm}\label{decomposition}  Let $(G,\Sigma,S)$ be a second countable locally compact twisted groupoid extension with Haar systems. Then the twisted groupoid C*-algebra $C^*(G,\Sigma)$ is isomorphic to the crossed product C*-algebra $C^*(\Sigma, \Sigma_{|S}, {\mathcal C}^*(S,\Sigma_{|S}))$ of the dynamical system $(\Sigma, \Sigma_{|S}, {\mathcal C}^*(S,\Sigma_{|S}))$.
 \end{thm}

\begin{proof} The proof is essentially the same as in the untwisted case  \cite[Theorem 2.11]{ikrsw:extensions}. Consider the map 
 \[j: C_c(G,\Sigma)\to C_c(\Sigma, \Sigma_{|S}, {\mathcal C}^*(S,\Sigma_{|S}))\]
 which sends $f\in  C_c(G,\Sigma)$ to $g=j(f)$ where
 \[g(\tau)(\sigma)=\delta^{1/2}(p(\tau))f(\sigma\tau),\qquad (\sigma,\tau)\in \Sigma_{|S}*\Sigma.\]
It can be shown exactly as in  \cite[Section 2]{ikrsw:extensions} that it is a $*$-algebra homomorphism, continuous in the inductive limit topology and with dense range in in the inductive limit topology. Therefore it extends to a surjective $*$-homomorphism
\[j: C^*(G,\Sigma)\to C^*(\Sigma, \Sigma_{|S}, {\mathcal C}^*(S,\Sigma_{|S})).\]
Let $(\mu,{\mathcal H})$ be a representation of $(G,\Sigma)$ defining respectively the representation $L$ of $C^*(G,\Sigma)$ and  the representation $\underline L$ of $ C^*(\Sigma, \Sigma_{|S}, {\mathcal C}^*(S,\Sigma_{|S}))$. We show that $L=\underline L\circ j$. For a better legibility, we write $\dot\tau=p(\tau)$ for $\tau\in\Sigma$. Given $\xi\in L^2(G^{(0)},\mu,{\mathcal H})$ and $f\in C_c(G,\Sigma)$, we have a.e. $x$
\begin{align*}
L(f)\xi(x)
&=\int_G f(\tau)\, [\tau\xi(s(\tau))] \,\Delta^{-1/2}(\dot\tau)d\lambda^x(\dot\tau)\\
&=\int_H\int_S f(\tau\sigma) \, [\tau\sigma\xi(s(\tau))] \, \Delta^{-1/2}(\dot\tau\dot\sigma)d\beta^{s(\tau)}(\dot\sigma)d\alpha^x(\pi_\Sigma(\tau))\\ 
\end{align*}
On the other hand,
\begin{align*}
 \underline L(j(f))\xi(x)
 &=\int_H j(f)(\tau) \, [\tau\xi(s(\tau))] \,\underline\Delta^{-1/2}(\pi_\Sigma(\tau))d\alpha^x(\pi_\Sigma(\tau))\\
 &=\int_H\int_S j(f)(\tau)(\sigma) \, [\sigma\tau\xi(s(\tau))] \,\delta^{-1/2}(\dot\sigma)d\beta^{r(\tau)}(\dot\sigma)\underline\Delta^{-1/2}(\pi_\Sigma(\tau))d\alpha^x(\pi_\Sigma(\tau))\\
 &=\int_H\int_S\delta^{1/2}(\dot\tau)f(\sigma\tau) \, [\sigma\tau\xi(s(\tau))] \,\delta^{-1/2}(\dot\sigma)d\beta^{r(\tau)}(\dot\sigma)\underline\Delta^{-1/2}(\pi_\Sigma(\tau))d\alpha^x(\pi_\Sigma(\tau))\\
  &=\int_H\int_S\delta^{-1/2}(\dot\tau)f(\tau\sigma) \, [\tau\sigma\xi(s(\tau))] \,\delta^{-1/2}(\dot\sigma)d\beta^{s(\tau)}(\dot\sigma)\underline\Delta^{-1/2}(\pi_\Sigma(\tau))d\alpha^x(\pi_\Sigma(\tau))\\
  &=\int_H\int_Sf(\tau\sigma) \, [\tau\sigma\xi(s(\tau))] \,\delta^{-1/2}(\dot\tau\dot\sigma)\underline\Delta^{-1/2}(\pi_\Sigma(\tau))d\beta^{s(\tau)}(\dot\sigma)d\alpha^x(\pi_\Sigma(\tau))\\
  &=\int_H\int_Sf(\tau\sigma) \, [\tau\sigma\xi(s(\tau))] \,\Delta^{-1/2}(\dot\tau\dot\sigma)d\beta^{s(\tau)}(\dot\sigma)d\alpha^x(\pi_\Sigma(\tau))\\
 \end{align*}
This shows that $L=\underline L\circ j$. From the definition of the full norms, we deduce that $j$ is isometric.  Therefore, $j$ extends to a $*$-isomorphism
\[j: C^*(G,\Sigma)\to C^*(\Sigma, \Sigma_{|S}, {\mathcal C}^*(S,\Sigma_{|S})).\]
\end{proof}

In the case of an untwisted groupoid, one obtains

\begin{cor} Let $G$ be a second countable locally compact groupoid with Haar system and let $S\subset G$ be a closed normal subgroupoid with Haar system.  Then the groupoid C*-algebra $C^*(G)$ is isomorphic to the crossed product C*-algebra $C^*(G, S, {\mathcal C}^*(S))$ of the dynamical system $(G, S, {\mathcal C}^*(S))$.
\end{cor}
In the case of a group, one recovers Green's Theorem \ref{Green} given in the introduction.

\section{The abelian case}

When the normal subgroupoid $S$ in the locally compact twisted groupoid extension with Haar systems $(G,\Sigma,S)$ is abelian and the restriction $\Sigma_{|S}$ of the twist is also abelian, one can go one step further by using the Gelfand transform for the bundle of abelian C*-algebras ${\mathcal C}^*(S,\Sigma_{|S}))$. Let us first consider a general abelian groupoid C*-dynamical system.

\subsection{Abelian groupoid C*-dynamical system} 

\begin{defn} We say that a groupoid dynamical system $(G, S, {\mathcal A})$ is abelian when $\mathcal A$ is a continuous bundle of commutative C*-algebras over $G^{(0)}$.
 \end{defn}
This is not Green's definition of an abelian system as in \cite[Section 7]{gre:local}, which means that $H=G/S$, instead of $\mathcal A$, is abelian. Note also that we do not assume here that the groups $S_x$ are abelian.  The sectional C*-algebra $A=C_0(G^{(0)},{\mathcal A})$ is abelian, hence isomorphic to $C_0(Y)$, where $Y$ is the spectrum of $A$. The space $Y$ is fiberd above $G^{(0)}$. The bundle map, denoted by $s:Y\to G^{(0)}$, is continuous and open. The fiber above $x\in G^{(0)}$ is written $Y_x$. We write $A=C_0(Y)$ and $A_x=C_0(Y_x)$. The action of $G$ on ${\mathcal A}$ induces an action on $Y$ which we write as a right action so that $(\gamma.f)(y)=f(y\gamma)$ where $f\in C_0(Y_{s(\gamma)})=A_{s(\gamma)}$ and $y\in Y_{r(\gamma)}$. Because the action of $S$ is unitarily implemented and $\mathcal A$ is abelian, $S$ acts trivially on $\mathcal A$. Therefore, the action of $G$ is in fact an action of $H=G/S$. The semi-direct product groupoid of the action is denoted by $Y\rtimes H$. It carries a twist $\Sigma$, given by a pushout construction. We first observe that the homomorphism $\chi: S\to UM({\mathcal A})$ which implements the restriction of the action to $S$ gives a map 
$${\underline\chi}:Y*S\to \t\qquad (y,t)\mapsto (\chi(t))(y)$$
It is a continuous groupoid homomorphism which satisfies ${\underline\chi}(y,\gamma t \gamma^{-1})={\underline\chi}(y\gamma,t)$ for all $(\gamma,t)\in G*S$. Here is a general definition.

\begin{defn} Given a groupoid extension $S\rightarrowtail G \twoheadrightarrow H$ and an $H$-bundle of abelian groups $T$, we say that a group bundle morphism $\varphi : S\to T$ is equivariant if $\varphi(\gamma s\gamma^{-1})=\dot\gamma \varphi(s)$ for all $(\gamma,s)\in G*S$, where $\dot\gamma$ is the image of $\gamma$ in $H$.
\end{defn}

We give now the general pushout construction (the reader is directed to \cite{ikrsw:pushouts} for a full exposition). It is summarized by the following diagram.
\[\begin{CD}
S@>{}>{}>G @>{}>{}>H\\
 @V{\varphi}V{}V  @V{\varphi_{*}}V{}V@|\\
T @>{}>{}>\underline G@>{}>{}> H
\end{CD}\]

Here are the details.

\begin{prop} Let $S\rightarrowtail G \twoheadrightarrow H$ be a groupoid extension, let $p_T:T\to G^{(0)}$ be a locally compact abelian group
bundle endowed with an $H$-action and let $\varphi:S\to T$ be an equivariant group bundle morphism. Then there is an extension
$T\rightarrowtail\underline G \twoheadrightarrow H$ and a morphism $\varphi_*: G\to \underline G$
that is compatible with $\varphi$. They are unique up to isomorphism.
\end{prop}
\begin{proof}
We  define 
  \[
   T*  G=\{\,(t,\gamma )\in
       T\times  G,
    | \,p_T(t)=r(\gamma)\,\}
  \]
  It is a groupoid over $\go$ with multiplication
  \[(t,\gamma)(t',\gamma')=(t(\dot\gamma t'), \gamma\gamma')\]
  and inverse
  \[(t,\gamma)^{-1}=(\dot\gamma^{-1}(t^{-1}),\gamma^{-1})\]
  Endowed with the relative topology, it is a locally compact topological groupoid. Then $S$ embeds into it as a closed normal subgroupoid via $i:S\to T*  G$ given by $i(s)=(\varphi(s^{-1}), s)$. We define $\underline G:=T*  G/i(S)$. Equivalently, $\underline G$ is the quotient of  $T*  G$ for the left action of $S$ given by
  $s(t,\gamma)=(t\varphi(s^{-1}), s\gamma)$.
   Its elements are of the form $[t,\gamma]$ where $(t,\gamma)\in T*  G$ and satisfy $[t,\gamma]=[t\varphi(s^{-1}), s\gamma]$ for $s\in S_{r(\gamma)}$. Let us spell out its groupoid structure. Its unit space is $G^{(0)}$ with obvious range and source maps. The multiplication is given by
    \[[t,\gamma][t',\gamma']=[t(\dot\gamma t'), \gamma\gamma']\]
 and its inverse  map is given
  \[[t,\gamma]^{-1}=[\dot\gamma^{-1}(t^{-1}),\gamma^{-1}]\]
The map $\underline \pi: \underline G\to H$ given by $\underline \pi[t,\gamma]=\pi(\gamma)$ is a surjective homomorphism and $\underline {\pi}^{(0)}$ is the identity map. Its kernel is identified to $T$ via the map $\underline j:T\to\underline G$ defined by $\underline j(t)=[t,p_T(t)]$. The map $\varphi_*: G\to\underline G$ is given by $\varphi_*(\gamma)=[r(\gamma),\gamma]$ for $\gamma\in G$.
  
\end{proof}

\begin{defn} The above extension $T\rightarrowtail\underline G \twoheadrightarrow H$ is called the pushout of the extension $S\rightarrowtail  G \twoheadrightarrow H$ by the morphism $\varphi:S\to T$.
 
\end{defn}

To apply this construction to our abelian groupoid C*-dynamical system $(G,S,{\mathcal A})$, we first consider the extension
 \[S\rightarrowtail G\twoheadrightarrow H\]
Taking the semi-direct product, we obtain a new extension
  \[Y*S\rightarrowtail Y\rtimes G\twoheadrightarrow Y\rtimes H\]
We view $Y\times\t$ as a group bundle over $Y$ with the trivial action of $Y\rtimes H$. The map 
  \[\varphi:Y*S\to Y\times\t\quad\hbox{given by}\quad \varphi(y,t)=(y,{\underline\chi}(y,t))\]  is a group bundle morphism which is equivariant in the above sense. We define the extension 
  \[Y\times \t\rightarrowtail \Sigma \twoheadrightarrow Y\rtimes H\]
  as the pushout by this morphism. Explicitly,
\[\Sigma=\{[\theta,y,\gamma]: \theta\in\t, (y,\gamma)\in Y\rtimes G\}\]
where
\[[\theta,y,t\gamma]=[\theta\, {\underline\chi}(y,t),y,\gamma],\quad \forall (t,\gamma)\in S*G\quad.\]

Note that ${\Sigma}$ is a twist over the semi-direct product $Y\rtimes H$. We now use the same method as in the proof of Theorem \ref{decomposition} to show that the dynamical systems $(Y\rtimes H, \Sigma)$ and $(G,S, {\mathcal A})$ have isomorphic C*-algebras: we first show that they have the same representations.

\begin{prop}\label{representation2} Assume that $G$ is second countable and that $\mathcal A$ is separable. There is a natural one-to-one correspondence between the representations of $(Y\rtimes H, \Sigma)$ and those of $(G,S, {\mathcal A})$.
 \end{prop}

\begin{proof}
 A representation of $(Y\rtimes H, \Sigma)$ is a pair $(\mu,{\mathcal H})$, where $\mu$ is a measure on $Y$ quasi-invariant with respect of $(H,\alpha)$ and ${\mathcal H}$ is a measurable Hilbert bundle over  $Y$ endowed with a unitary representation of $\Sigma$ such that $(\theta,y)\xi=\theta\xi$ for all $(\theta,y,\xi)\in\t\times(Y*{\mathcal H})$. On the other hand, a representation of $(G,S, {\mathcal A})$ is a pair $(\underline\mu,\underline{\mathcal H})$, where $\underline\mu$ is a measure on $G^{(0)}$ quasi-invariant with respect of $(H,\alpha)$ and $\underline{\mathcal H}$ is a measurable Hilbert bundle over $G^{(0)}$ endowed with a unitary representation of $G$ and a representation of ${\mathcal A}$ satisfying the conditions (i) and (ii) of Definition \ref{representation}. Let us sketch how to construct $(\underline\mu,\underline{\mathcal H})$ from $(\mu,{\mathcal H})$. We let $\underline\mu$ be pseudo-image of $\mu$ under the map $s:Y\to G^{(0)}$. The measure $\mu$ can be written as
 $\mu=\int \rho_xd\underline\mu(x)$ where $(\rho_x)_{x\in G^{(0)}}$ is a family of measures along the fibers of $s$. According to \cite[Corollary 5.3.11]{adr:amenable} or \cite[Proposition 3.1]{ren:hyper}, the quasi-invariance of $\mu$ is equivalent to the quasi-invariance of $\underline\mu$ and the equivalence $\rho_{r(h)}.h\sim \rho_{s(h)}$ for a.e. $h$. Moreover, the respective Radon-Nikodym derivatives $\Delta$ of $\mu$ and $\underline\Delta$ of $\underline\mu$ are related by
 \[\Delta(y,h)\tau(y,h)=\underline\Delta(h)\qquad{\rm where}\qquad \rho_{s(h)}.h^{-1}=\tau( ., h)\rho_{r(h)}.\]
We define for $x\in G^{(0)}$ the Hilbert space $\underline{\mathcal H}_x=L^2(Y_x,\rho_x; {\mathcal H}_{|Y_x})$. A fundamental family of square-integrable sections of ${\mathcal H}$ provides a fundamental family of measurable sections of $\underline{\mathcal H}=(\underline{\mathcal H}_x)_{x\in G^{(0)}}$. Given $(\gamma,\xi)\in G*\underline{\mathcal H}$, we define
\[(\gamma\xi)(y)=\tau(y,\dot\gamma)^{1/2}[1,y,\gamma]\xi(y\dot\gamma)\quad{\rm for}\quad y\in Y_{r(\gamma)}.\]
One can check that it is a unitary representation of $G$.
Given $(a,\xi)\in C_0(Y_x)\times{\mathcal H}_x$, we define
\[(a\xi)(y)=a(y)\xi(y) \quad{\rm for}\quad y\in Y_x.\]
It is a representation of ${\mathcal A}$ satisfying the conditions (i) and (ii) of Definition \ref{representation}. Conversely, let $(\underline\mu,\underline{\mathcal H})$ be a representation of $(G,S, {\mathcal A})$. It defines in particular a representation $M$ of $C_0(Y)$ on $L^2(G^{(0)},\underline\mu,\underline{\mathcal H})$ such that for $a\in C_0(Y)$ and $\xi\in L^2(G^{(0)},\underline\mu,\underline{\mathcal H})$, $(M(a)\xi)(x)=a_{|Y_x}\xi(x)$ for $\underline\mu$-a.e. $x$. By the representation theory of commutative C*-algebras, there exists a measure $\mu$ on $Y$ and
a measurable Hilbert bundle ${\mathcal H}$ over $Y$ such that $M$ is unitarily equivalent to the representation by multiplication operators on $L^2(Y,\mu,{\mathcal H})$. One can check that $\underline\mu$ is a pseudo-image under $s:Y\to G^{(0)}$ of $\mu$. We write $\mu$ as $\mu=\int \rho_xd\underline\mu(x)$. This gives the disintegration
$L^2(Y,\mu,{\mathcal H})=L^2(G^{(0)},\underline\mu,\underline{\mathcal H}')$ where $\underline{\mathcal H}'_x=L^2(Y_x,\rho_x, {\mathcal H}_{|Y_x})$. Since the unitary operator implementing the unitary equivalence of the representations commutes with the operators of multiplication by $h\in L^\infty(G^{(0)},\underline\mu)$, the Hilbert bundles $\underline{\mathcal H}$ and $\underline{\mathcal H}'$ are isomorphic. From now on, we assume that $\underline{\mathcal H}=\underline{\mathcal H}'$ and that $C_0(Y)$ acts by multiplication. According to \cite[Proposition 1, page 82]{gui:discrete} and condition (i), the unitary operator \[L(\gamma): L^2(Y_{s(\gamma)},\rho_{s(\gamma)}, {\mathcal H}_{|Y_{{s(\gamma)}}})\to L^2(Y_r(\gamma),\rho_{r(\gamma)}, {\mathcal H}_{|Y_{{r(\gamma)}}})\]
is of the form
\[L(\gamma)\xi (y)=\tau(y,\dot\gamma)^{1/2} u(y,\gamma)\xi(y\dot\gamma)\quad{\rm for}\quad y\in Y_{r(\gamma)}\]
where $\tau(.,h)$ is the Radon-Nikodym derivative  $d\rho_{s(h)}.h^{-1}/d\mu_{r(h)}$ and $u(y,\gamma)$ is a unitary operator from ${\mathcal H}_{y\dot\gamma}$ to ${\mathcal H}_y$ defined for a.e. $y$. Note that this implies that the measures $\rho_{s(h)}.h^{-1}$ and $\rho_{r(h)}$ are equivalent. As seen earlier, together with the quasi-invariance of $\underline\mu$, this implies the quasi-invariance of $\mu$ with respect to $(H,\alpha)$. Then, one can check using condition (ii) that 
\[[\theta,y,\gamma]\xi=\theta u(y,\gamma)\xi\qquad{\rm for}\qquad [\theta,y,\gamma]\in\Sigma \quad {\rm and}\quad \xi\in{\mathcal H}_{y\dot\gamma}\]
is well-defined and defines a unitary representation of $\Sigma$ on $\mathcal H$.
\end{proof}

The isomorphism of the next theorem can be viewed as a partial Gelfand transform. When $S=G^{(0)}$, then $G=H$, the twist $\Sigma$ is trivial and the theorem reduces to the well-known isomorphism of $C^*(H, {\mathcal A})$ and $C^*(Y\rtimes H)$.

\begin{thm}\label{abelian} Let $(G,S, {\mathcal A})$ a groupoid C*-dynamical system where ${\mathcal A}$ is abelian. Assume that $G$ is second countable and that $\mathcal A$ is separable. Let $Y$ be the spectrum of the abelian C*-algebra $C_0(\go,{\mathcal A})$. Then the twisted crossed product $C^*(G,S, {\mathcal A})$ is isomorphic to $C^*(Y\rtimes H, \Sigma)$, where $\Sigma$ is the above twist. 
\end{thm}

\begin{proof} There is a natural map $j$ from $C_c(Y\rtimes H, \Sigma)$ to $C_c(G,S, {\mathcal A})$ which we describe now. An element of $C_c(Y\rtimes H,\Sigma)$  is a continuous function $f:\Sigma\to\C$ which is compactly supported modulo $\t$ (since $\t$ is compact, this equivalent to be compactly supported) and which satisfies
$f[\theta'\theta,y,\gamma]=f[\theta,y,\gamma]\theta'^{-1}$ for $\theta,\theta'\in\t$ and $(y,\gamma)\in Y\rtimes G$. It is completely determined by its restriction to $\theta=1$. Let us set  $g(y,\gamma)=f[1,y,\gamma]$. This is a complex-valued function defined on $Y\rtimes G$ which is continuous with compact support modulo $Y\rtimes S$ and which satisfies
\[g(y,s\gamma)=g(y,\gamma){\underline\chi}(y,s)^{-1}\qquad\forall (s,y,\gamma)\in S*Y\rtimes G\] 
We define $j(f)$ by $j(f)(\gamma)(y)=f[1,y,\gamma]$. It is a continuous section of the bundle $r^*{\mathcal A}$ over $G$, it has a compact support modulo $S$ and it satisfies $j(f)(s\gamma)=j(f)(\gamma)\chi(s^{-1})$ for all $(s,\gamma)\in S*G$. Thus it is an element of $C_c(G,S, {\mathcal A})$. It is clear that the map $j$ is injective. Since the inductive limit topology on $C_c(Y_x)\subset C_0(Y_x)$ is finer than the topology of uniform convergence, it is continuous with respect to the inductive limit topology. Let us show that $j$ is a $*$-homomorphism. Writing its elements $\underline f, \underline g$ as complex-valued functions on $Y\rtimes G$, the $*$-algebraic structure of $C_c(G,S, {\mathcal A})$  is given by
\[\underline f*\underline g(y,\gamma)=\int \underline f(y,\tau)\underline g(y\tau,\tau^{-1}\gamma)d\alpha^{r(\gamma)}(\dot\tau),\qquad \underline f^*(y,\gamma)=\overline{{\underline f}(y\gamma,\gamma^{-1})}\]
This makes clear that $j$ is a $*$-homomorphism. One checks also easily that $j$ is continuous when $C_c(Y\rtimes H, \Sigma)$ and $C_c(G,S, {\mathcal A})$ are endowed with the inductive limit topology. Moreover, $j$ has dense range: we return to the original description of $C_c(G,S, {\mathcal A})$ as the space of  compactly supported continuous sections of a Banach bundle $\mathcal B$ over $H$. Then $j(C_c(Y\rtimes G, \Sigma))$ is a linear subspace of $C_c(H,{\mathcal B})$ which satisfies conditions (I) and (II) of \cite[Proposition 14.6]{fd:Fell}. Indeed, for $\lambda$ continuous function on $H$ and $f\in j(C_c(Y\rtimes G, \Sigma))$, the function $\lambda f$ defined by 
\[(\lambda f)(y,\gamma)=\lambda(\dot\gamma)f(y,\gamma)\qquad\forall (y,\gamma)\in Y\rtimes G\] 
belongs to $C_c(Y\rtimes G, \Sigma)$. The fiber $B_h$ of the bundle $\mathcal B$ can be identified with the Banach space $C_0(Y_{r(h)})$. In this identification, the evaluation at $h$ of the elements of $C_c(Y\rtimes G, \Sigma)$ gives the whole subspace $C_c(Z_{r(h)})$, which is dense in $C_0(Z_{r(h)})$. This shows that $j(C_c(Z\rtimes G, \Sigma))$ is dense in $C_c(G,S, {\mathcal A})$ in the inductive limit topology. By definition of the full norm, $j$ extends to a $*$-homomorphism from $C^*(Z\rtimes H, \Sigma)$ onto $C^*(G,S, {\mathcal A})$. To show that it is injective, it suffices to show that for all non-degenerate representation $L$ of $C_c(Z\rtimes G, \Sigma)$, there exists a representation $\underline L$ of $C_c(G,S, {\mathcal A})$ such that $L=\underline L\circ j$. We can assume that $L$ is the integrated representation of $(\mu, {\mathcal H})$ as in the above proposition. We denote by $(\underline\mu, \underline{\mathcal H})$ the corresponding representation of $(G,S, {\mathcal A})$ and by $\underline L$ its integrated representation. The Hilbert spaces of these representations are naturally identified, namely $L^2(Z,\mu,{\mathcal H})=L^2(G^{(0)},\underline\mu,\underline{\mathcal H})$. Let us show that  $L=\underline L\circ j$. Given $\xi,\eta\in L^2(G^{(0)},\underline\mu,\underline{\mathcal H})$, we have:
\begin{align*}
(\xi\,|\,\underline L\circ j(f)\eta)
&=\int (\xi\circ r(\gamma)\,|\,j(f)(\gamma)(\gamma\eta\circ s(\gamma)))\underline\Delta^{-1/2}(\dot\gamma)d\alpha^x(\dot\gamma)d\underline\mu(x) \\
&=\int\int (\xi(y)\,|\,f[1,y,\gamma]\tau(y,\dot\gamma)^{1/2}[1,y,\gamma]\eta(y\dot\gamma))d\rho_{r(\gamma)}(y)\underline\Delta^{-1/2}(\dot\gamma)d\alpha^x(\dot\gamma)d\underline\mu(x) \\
&=\int\int f[1,y,\gamma] (\xi(y)\,|\,[1,y,\gamma]\eta(y\dot\gamma))\tau(y,\dot\gamma)^{1/2}\underline\Delta^{-1/2}(\dot\gamma)d\alpha^{s(y)}(\dot\gamma)d\rho_x(y)d\underline\mu(x) \\
&=\int f[1,y,\gamma] (\xi(y)\,|\,[1,y,\gamma]\eta(y\dot\gamma))\Delta^{-1/2}(y,\dot\gamma)d\alpha^{s(y)}(\dot\gamma)d\mu(y) \\
&=(\xi\,|\,L(f)\eta)\\
\end{align*}
\end{proof}

\begin{rem}
 Just as in \cite[Proposition 3.2]{ikrsw:extensions}, a shorter proof avoiding the explicit form of the integrated representations can be obtained by observing that the above $*$-homomorphism $j$ is isometric with respect to the I-norm. We have chosen this presentation because the explicit correspondence given in Proposition \ref{representation2} between the representations of the two systems has its own interest.
\end{rem}

\subsection{Abelian twisted extensions}

We return now to our initial problem, which is the analysis of a twisted groupoid C*-algebra $C^*(G,\Sigma)$ in presence of a closed normal subgroupoid $S$ having a Haar system. As said earlier, we make a further assumption, whose present form I owe to Alex Kumjian.
\begin{defn}\label{abelian extension} We say that a twisted extension $(G,\Sigma, S)$ is abelian if
\begin{enumerate}
\item  the group bundle $S$ is abelian and
\item  the group bundle $\Sigma_{|S}$ is abelian.
\end{enumerate}
\end{defn}
Condition (i) is stated for convenience only since it is implied by condition (ii). When the twisted extension is abelian, the bundle of C*-algebras  ${\mathcal C}^*(S,\Sigma_{|S}))$ is abelian.  Since the groups $S_x$ are amenable, it is a continuous bundle of abelian C*-algebras (see \cite[Theorem 5.5]{lr:quantization}). Thus, $(\Sigma,\Sigma_{|S}, {\mathcal C}^*(S,\Sigma_{|S}))$ is an abelian groupoid dynamical system in the above sense. We combine \thmref{decomposition} and \thmref{abelian} to obtain our main result.

\begin{thm}\label{main}  Let $(G,\Sigma,S)$ be a locally compact abelian twisted groupoid extension.  Assume that $G$ is second countable. Then the twisted groupoid C*-algebra $C^*(G,\Sigma)$ is isomorphic to the twisted groupoid C*-algebra $C^*(Y\rtimes H, \underline\Sigma)$ where $Y$ is the spectrum of $C^*(S,\Sigma_{|S}))$, $H=G/S$ and the twist $\underline\Sigma$ is obtained by a pushout construction. 
\end{thm}

\begin{rem}
 The interest of this result is that it diminishes the isotropy. In the rather exceptional case when the isotropy bundle itself $S=G'$ has a (continuous!) Haar system and satisfies the assumptions of \defnref{abelian extension}, then $H=G/S$ is a principal groupoid and so is $Y\rtimes H$.
\end{rem}

For applications, it is necessary to be more explicit about the space $Y$, the action of $H$ on it and the twist $\underline\Sigma$. Recall that we assume that the locally compact abelian group bundle $S$ has a Haar system. Therefore, we endow its dual group bundle $\hat S$ with the topology of the spectrum of $C^*(S)$, as in \cite[Section 3] {mrw:continuous-trace III}. Then the abelian group bundle $\Sigma_{|S}$, as an extension of $S$ by $G^{(0)}\times\t$, has a Haar system and its dual group bundle $\hat\Sigma_{|S}$ has also a natural locally compact topology.  An element of $\hat S$ [resp. $\hat\Sigma_{|S}$]  will be denoted by $(x,\chi)$, where $x$ is a base point and $\chi\in\hat S_x$ [resp. $\chi\in(\hat\Sigma_{|S})_x$] . We first consider the general case of an abelian twist $(S,\Sigma)$.
\begin{defn} Let $X\times\t\rightarrowtail\Sigma\twoheadrightarrow S$  be an abelian twist over an abelian group bundle $S$. Its twisted spectrum is defined as
 $$\hat S^\Sigma=\{(x,\chi) \in\hat\Sigma\quad\hbox{such that}\quad   \chi(\theta)=\theta\quad\forall\theta\in \t \}$$
 \end{defn}
 
 An element $(x,\chi)$ of the twisted spectrum defines a trivialization of the twist $(S_x,\Sigma_x)$. We shall see soon that these trivializations always exist when $S_x$ is locally compact. The twisted spectrum $\hat S^\Sigma$ is an affine space over the dual group bundle $\hat S$: the action of $\hat S$ on $\hat S^\Sigma$ is the usual multiplication: given $\chi\in \hat S_x^\Sigma$ and $\rho\in\hat S_x$, $\chi\rho\in\hat S_x^\Sigma$ is defined by $(\chi\rho)(\sigma)=\chi(\sigma)\rho(\dot\sigma)$ for $\sigma\in\Sigma_x$ and where $\dot\sigma$ is the image of $\sigma$ in $S_x$.  
 
\begin{rem}\label{trivial}  In \cite{dgn:Weyl}, the authors give a similar description of the twisted spectrum when the twist $\Sigma$ is given by a symmetric 2-cocycle. Then, the twist is obviously abelian. It may be useful to recall here that, according to \cite[Lemma 7.2]{kle:multipliers}, a Borel 2-cocycle on a locally compact abelian group is trivial if and only if it is symmetric. Moreover, if the topology of the group is second countable, every twist is given by a Borel 2-cocycle. Thus, a twist over a bundle of second countable locally compact abelian groups is abelian if and only if it is pointwise trivial. This result can also be deduced from the following lemma.
\end{rem}

\begin{lem}\label{spectrum} Let $(S,\Sigma)$ be an abelian twist over a  locally compact  abelian group bundle $S$.
 Assume that $S$ has a Haar system $\beta$. Then,
 \begin{enumerate}
 \item  the spectrum of the abelian C*-algebra $C^*(S,\Sigma)$ is  the twisted spectrum $\hat S^\Sigma$;
 \item  $\hat S^\Sigma$ is a locally compact space endowed with a free, proper and transitive action of the dual bundle $\hat S$.
\end{enumerate}
\end{lem}

\begin{proof} According to  \cite[Lemma 2.4]{ren:ideal} (see also \cite[Section 1]{gre:local}), the map $\Theta: C_c(\Sigma)\to C_c(S,\Sigma)$ such that
\[\Theta(F)(\sigma)=\int_\t F(\theta\sigma)\theta d\theta\]
extends to a $*$-homomorphism from $C^*(\Sigma)$ onto $C^*(S,\Sigma)$. This gives an inclusion of the spectrum of $C^*(S,\Sigma)$ as a closed subset of the spectrum $\hat\Sigma$ of $C^*(\Sigma)$. We have already seen that an element of $\hat\Sigma$ is given by a pair $(x,\chi)$, where $x\in S^{(0)}$ and $\chi\in\hat\Sigma_x$. According to \cite[Lemma 2.4]{ren:ideal}, an element $(x,\chi)$ of $\hat\Sigma$ belongs to the spectrum of $C^*(S,\Sigma)$ if and only if it defines a representation of $(S,\Sigma)$, which means here that $\chi(\theta)=\theta$ for all $\theta\in\t$. This is the above definition of $\hat S^\Sigma$. The space $\hat S$ can also be identified topologically with a subset of $\hat\Sigma$ via the map $\rho\mapsto\rho\circ p$, where $p:\Sigma\to S$ is the projection. The action $\hat S^\Sigma*\hat S\to \hat S^\Sigma$ defined above is the restriction of the pointwise multiplication $\hat\Sigma*\hat\Sigma\to\hat\Sigma$. According to \cite[Corollary 3.4]{mrw:continuous-trace III}, this map is continuous.
\end{proof}

We return to our twisted abelian extension $(G,\Sigma, S)$. We denote by $\hat S^\Sigma$ rather than $\hat S^{\Sigma_{|S}}$ the twisted spectrum of $(S,\Sigma_{|S})$. Let us describe the action of $H=G/S$ on $\hat S^\Sigma$. The groupoid $H$ acts on the group bundle $\Sigma_{|S}$ by conjugation:
$h.\sigma=\tau\sigma\tau^{-1}$, where $\pi_\Sigma(\tau)=h$. The transposed action on the dual group bundle $\widehat{\Sigma_{|S}}$, defined by $(\chi h)(\sigma)=\chi(\tau\sigma\tau^{-1})$, preserves the twisted spectrum $\hat S^\Sigma$. It is easily checked that this is the action arising from the action of $H$ on the bundle of C*-algebras ${\mathcal C}^*(S,\Sigma_{|S})$.\\

The above pushout diagram defining the twist $\underline\Sigma$ becomes:

\[\begin{CD}
\hat S^\Sigma*\Sigma_{|S}@>{}>{}>\hat S^\Sigma\rtimes \Sigma @>{}>{}>\hat S^\Sigma\rtimes H\\
 @V{\varphi}V{}V  @V{}V{}V@|\\
\hat S^\Sigma\times \t @>{}>{}>\underline\Sigma@>{}>{}>\hat S^\Sigma\rtimes H
\end{CD}\]

where $\varphi(\chi,\sigma)=(\chi, \chi(\sigma))$ for $(\chi,\sigma)\in \hat S^\Sigma*\Sigma_{|S}$. \\
 Explicitly, 
$\underline\Sigma$ is the quotient of the groupoid $(\hat S^\Sigma\rtimes \Sigma)\times\t$ by the equivalence relation
$$(\chi,\sigma\tau, \theta)\sim (\chi, \tau, \theta\,\chi(\sigma)),\quad\forall\sigma\in\Sigma_{|S}.$$

\begin{rem} In \cite[Proposition 3.5]{ikrsw:pushouts}, the above \thmref{main} (the twisted case) is deduced from the similar result for the untwisted case, established in the previous work \cite[Theorem 3.3]{ikrsw:extensions}.
 
\end{rem}

\section{Examples and applications}

\subsection{Deformation quantization}

Rieffel has introduced a notion of C*-algebraic deformation quantization and illustrated it by a number of examples in \cite{rie:Heisenberg}. On the other hand, Ramazan, generalizing Connes' tangent groupoid, has produced deformation quantization of Lie-Poisson manifolds by using groupoid techniques (see \cite{ram:thesis, lr:quantization}). Our \thmref{main} shows that the two approaches agree on the basic example of a symplectic finite-dimensional real vector space $(V,\omega)$. Then, for every $\hbar\in\R$, $\sigma_\hbar=e^{i\hbar\omega/2}$ is
 a $\t$-valued 2-cocycle on $(V,+)$. It is shown in \cite{rie:fields} that $\hbar\mapsto C^*(V,\sigma_\hbar)$ can be made into a continuous field of C*-algebras and in \cite{rie:actions} that it gives a C*-algebraic deformation quantization of the Lie-Poisson manifold $(V,\omega)$.  
 
\begin{prop} Let $(V,\omega)$ be a symplectic finite-dimensional real vector space. Then the sectional C*-algebra of the continuous field $\hbar\mapsto C^*(V,\sigma_\hbar)$  is isomorphic to the C*-algebra of the tangent groupoid over a Lagrangian subspace of $V$.
\end{prop}

\begin{proof} We first observe that the sectional algebra can be viewed as a twisted groupoid C*-algebra $C^*(G,\Sigma)$, where $G$ is the trivial group bundle $\R\times V$ over $\R$ and $\Sigma$ is the twist defined by the 2-cocycle $\sigma(\hbar, .)=\sigma_\hbar$. Let  $V=L\oplus L'$ be a direct sum decomposition, where $L$ and $L'$ are complementary Lagrangian subspaces. This gives the extension
 $$\R\times L\rightarrowtail G\twoheadrightarrow \R\times L'$$
 The abelian group bundle $S=\R\times L$ satisfies the conditions of \defnref{twisted extension} with respect to the twisted groupoid $(G,\Sigma)$. Therefore, by \thmref{main}, $C^*(G,\Sigma)$ is isomorphic to $C^*(Y\rtimes (\R\times L'), \underline\Sigma)$, where $Y$ is the twisted spectrum and the twist $\underline\Sigma$ is obtained by the pushout construction. Let us determine them explicitly. The action of $H=\R\times L'$ on $\Sigma_{|S}=\R\times L\times\t$ is given by
 $$(\hbar,y).(\hbar,x,\theta)=(\hbar, x, e^{-i\hbar\omega(x,y)}\theta),\quad{\rm where}\quad\hbar\in\R,\,\quad y\in L',\, x\in L,\,  \theta\in\t$$
Since $\Sigma_{|S}=\R\times L\times\t$, the twisted spectrum $Y$ is $\R\times\hat L$, where $\hat L$ denotes the dual group of the abelian locally compact group $L$. The action of $H$ on $Y$ is given by
$$(\hbar,\chi)(\hbar, y)=(\hbar, \chi\varphi_\hbar(y)),\quad{\rm where}\quad \hbar\in\R,\,\chi\in\hat L,\, y\in L'$$
and for $\hbar\in\R$, $\varphi_\hbar$ is the group homomorphism from $L'$ to $\hat L$ such that
$$<\varphi_\hbar(y),x>=e^{-i\hbar\omega(x,y)}\quad{\rm where}\quad x\in L,\, y\in L'$$
The semi-direct product $Y\rtimes H$ is a bundle of semi-direct products $\hat L\rtimes_\hbar L'$.  For $\hbar\not=0$, $\varphi_\hbar$ is an isomorphism and $\hat L\rtimes_\hbar L'$ is isomorphic to the trivial groupoid $\hat L\times \hat L$. For $\hbar=0$, we get $\hat L\times L'$, where the first term is a space and the second is a group. We use again $\omega$ to identify $L'$ with the dual $L^*$, which is the tangent space of $\hat L$. Thus $\hat L\times L'$ is isomorphic to the tangent bundle $T\hat L$ and the groupoid $Y\rtimes H$ is isomorphic to the tangent groupoid of the manifold $\hat L$. One can check that we have an isomorphism of topological groupoids. To conclude, we show that the twist $\underline\Sigma$ is trivial.
By construction, $\underline\Sigma$ is the quotient of $(Y\rtimes\Sigma)\times\t$ by the equivalence relation
 $$(\hbar,\chi, x+v, \sigma_\hbar(x,v)\varphi\psi,\theta)\sim (\hbar,\chi,v,\psi,\theta\,\chi(x)\varphi)$$
 where $\hbar\in\R$, $\chi\in\hat L$, $x\in L, v\in V$ and $\varphi,\psi,\theta\in\t$. The map
 $$(Y\rtimes\Sigma)\times\t\to (Y\rtimes H)\times\t$$
 sending $(\hbar,\chi,x+y,\psi,\theta)$ to $(\hbar,\chi,y,\psi \chi(x)\sigma_\hbar(x,y)\theta)$ where $\hbar\in\R$, $\chi\in\hat L$, $x\in L$, $y\in L'$, and $\varphi,\theta\in\t$ identifies topologically this quotient. This is also a groupoid homomorphism. Therefore, $\underline\Sigma$ is isomorphic to $(Y\rtimes H)\times\t$.
  \end{proof}
  
 The above example can also be presented via the Heisenberg group ${\mathcal H}=V\times\R$ with multiplication
 $(v,s)(w,t)=(v+w, \omega(v,w)+s+t)$. Mackey's normal subgroup analysis, under the form of \thmref{main}, applied to the center $\{0\}\times\R$ gives the first deformation. The second deformation can be obtained by applying this analysis to the subgroup $L\times\R$. In conclusion, we have three isomorphic C*-algebras: $C^*({\mathcal H})$, $C^*(G,\Sigma)$ and $C^*(Y\rtimes H)$.
 
 \subsection{Twisted abelian groups}
The twisted abelian groups which appear in the previous example illustrate the following general theory. Green gives in \cite[Section 7]{gre:local} an elegant presentation of the main results of \cite{bk:multiplier} about multiplier representations of abelian groups. As noted by Kumjian in \cite[1:12 {\it Example} ]{kum:diagonals}, the introduction of twisted groupoids is quite natural in this context. We put here the theory of multiplier representations of abelian groups into our framework and deduce directly its main results.  We consider an arbitrary twisted group 
 $$\t\stackrel{j}{\rightarrowtail}\Sigma\stackrel{p}{\twoheadrightarrow}  G$$  where $G$ is a locally compact abelian group. We write $\dot\sigma=p(\sigma)$ and identify $\t$ with its image $j(\t)$. 
 The bicharacter of the twist is the function $b: G\times G\to \t$ such that 
 \[b(\dot\sigma,\dot\tau)=[\sigma,\tau]=\sigma\tau\sigma^{-1}\tau^{-1}.\]
 It defines a continuous group homomorphism $\varphi: G\to\hat G$, where $\hat G$ is the dual group, such that $\varphi(x)(y)=b(x,y)$ for all $x,y\in G$. We shall write $\hat x=\varphi(x)$ whenever it is convenient. The dual map $\hat\varphi: G\to\hat G$ is the complex conjugate: $\hat\varphi(x)=\overline{\varphi(x)}=\varphi(x)^{-1}$.

\begin{lem}\label{center} Let $(G,\Sigma)$ be a twisted abelian group. The following conditions are equivalent:
\begin{enumerate}
 \item the center $Z$ of the group $\Sigma$ is reduced to $\t$,
 \item $\varphi$ is one-to-one.
\end{enumerate}
 \end{lem}

\begin{proof} Indeed, by definition, the kernel of $\varphi$ is the image under $\pi$ of the center of $\Sigma$. 
\end{proof}
 
 \begin{defn}
 We say that the twisted abelian group $(G,\Sigma)$ is irreducible if the equivalent conditions of the previous lemma are satisfied. One also says that $b$ is totally skew.
 \end{defn}
 
 \begin{lem}\cite{kle:multipliers} Let $(G,\Sigma)$ be a twisted locally compact abelian group. The following conditions are equivalent:
\begin{enumerate}
 \item the group $\Sigma$ is abelian,
 \item $b$ is identically equal to one,
 \item the twist is trivial.
\end{enumerate}
 \end{lem}
  
\begin{proof} 
The equivalence of (i) and (ii) is clear. We have already mentioned the equivalence of (i) and (iii) in \remref{trivial}. Explicitly, it is clear that if the twist is trivial, then the group $\Sigma$ is abelian.  Conversely, if the group $\Sigma$ is abelian, \lemref{spectrum} shows that the twisted spectrum is not empty. Any of its elements gives a trivialization of the twist.
 \end{proof}
 
 \begin{rem}
One may compare this lemma with a classical result in the theory of abelian groups: an extension of abelian groups is trivial whenever its kernel is divisible (see for example \cite[Theorem 24.5]{fuc:abelian groups}).
\end{rem}

\begin{lem} Let $(G,\Sigma)$ be a twisted locally compact abelian group and let $S$ be a closed subgroup of $G$ such that $\Sigma_{|S}=p^{-1}(S)$ is abelian. 
\begin{enumerate}
\item There is a unique group homomorphism $\varphi_S: G/S \to \hat S$ which makes the following diagram commutative:
\[\begin{CD}
S@>{i}>{}>G@>{\pi}>{}>G/S\\
 @V{}V{\varphi_{|S}}V  @V{}V{\varphi}V@V{}V{\varphi_S}V\\
 \widehat{G/S}@>{\hat\pi}>{}>\hat G@>{\hat\i}>{}>\hat S
\end{CD}
\eqno(*)\]
where $\varphi_{|S}$ is the restriction of $\varphi$ to $S$ and  $\widehat{G/S}$ is identified with $S^\perp\subset \hat G$.
\item For all $(h,s)\in G/S\times S$, we have $\varphi_S(h)(s)=[\tau,\sigma]$, where $\tau$ is a pre-image of $h$ for the quotient map $\pi_\Sigma:\Sigma\to \Sigma/\Sigma_{|S}=G/S$ and $\sigma$ is a pre-image of $s$ for the map $p:\Sigma_{|S}\to S$.
\item The dual map $\hat\varphi_S$ is $\hat\varphi_{|S}$. 
\end{enumerate}
 
 \end{lem}

\begin{proof} Since $\Sigma_{|S}=p^{-1}(S)$ is abelian, $\varphi(S)$ is contained in $S^\perp$ and $\varphi$ passes to the quotients. The formula in (ii) results from the definition of $\varphi$. By definition, for $s=p(\sigma)\in S$ and $h=\pi_\Sigma(\tau)\in G/S$, we have
\[\hat\varphi_S(s)(h)=\varphi_S(h)(s)=[\tau,\sigma]=\overline{[\sigma,\tau]}=\overline{\varphi(s)(h)}=\hat\varphi(s)(h).\]
\end{proof}

The condition of being maximal abelian in the next lemma, which is prompted by the Lagrangian condition of the previous section, appears already in \cite[1.12. Example]{kum:diagonals}. It is also an assumption made in \cite[Theorem 3.1]{dgnrw:non-principal}.

\begin{lem}\label{map} Let $(G,\Sigma)$ be a twisted abelian group and let $S$ be a subgroup of $G$ such that $\Sigma_{|S}=p^{-1}(S)$ is abelian. Let  $\varphi_S: G/S \to \hat S$ be as above. Then
\begin{enumerate}
 \item $\varphi_S$ is one-to-one if and only if $\Sigma_{|S}$ is maximal abelian in $\Sigma$;
 \item $\varphi_S$ has dense range if and only if $\Sigma_{|S}\cap Z=\t$, where $Z$ is the center of $\Sigma$;
 \item $\varphi_S$ is an isomorphism and $Z=\t$ if and only if $\varphi$ is an isomorphism and $\Sigma_{|S}$ is maximal abelian in $\Sigma$.
 \end{enumerate}
 \end{lem}

\begin{proof} For (i), we compute the kernel ${\rm Ker}(\varphi_S)$: an element $h=\pi_\Sigma(\tau)$ belongs to ${\rm Ker}(\varphi_S)$ if and only if $[\tau,\sigma]=1$ for all $\sigma\in \Sigma_{|S}$. Thus, ${\rm Ker}(\varphi_S)$ is the image under $\pi_\Sigma$ of the commutant $\Sigma_{|S}'$ of $\Sigma_{|S}$ in $\Sigma$. It is reduced to 1 if and only if $\Sigma_{|S}'=\Sigma_{|S}$. For (ii), we consider the dual map $\hat\varphi_S: S\to  \widehat{G/S}$. It is defined by a similar expression:
$\hat\varphi_S(s)(h)=[\tau,\sigma]$, where $\tau$ is a pre-image of $h$ under $\pi_\Sigma$ and $\sigma$ is a pre-image of $s$ under $p:\Sigma_{|S}\to S$. An element $s=p(\sigma)$ belongs to ${\rm Ker}(\hat\varphi_S)$ if and only if $[\tau,\sigma]=1$ for all $\tau\in\Sigma$. Thus ${\rm Ker}(\hat\varphi_S)$ is the image under $p$ of $\Sigma_{|S}\cap Z$. It is reduced to 1 if and only if $\Sigma_{|S}\cap Z=\t$. This gives (ii) since the injectivity of $\hat\varphi_S$ is equivalent to the denseness of the range of $\varphi_S$. For (iii), we use the commutative diagram $(*)$ of the previous lemma. Suppose that $\varphi$ is an isomorphism. Since $\varphi$ and $\hat i$ are onto and open, so is $\varphi_S$. If moreover, $\Sigma_{|S}$ is maximal abelian in $\Sigma$, then, by (i), $\varphi_S$ is one-to-one. Therefore it is an isomorphism. Conversely, suppose that $\varphi_S:G/S\to\hat S$ is an isomorphism. Let us show that $\varphi$ is a proper map. Let $K$ be a compact subset of $\hat G$ such that $K=K^{-1}$. There is a compact subset $L$ of $G$ such that $\pi(L)=\varphi_S^{-1}(\hat\i(K))$. Let  $s\in S$ be such that $i(s)L\cap\varphi^{-1}(K)\not=\emptyset$. Then, $\hat\pi\circ\hat\varphi_S(s)=\hat\varphi\circ i(s)$ belongs to $K^2$. Since $\hat\pi$ is proper and $\hat\varphi_S$ is an isomorphism, the set $M$ of such elements $s$ is relatively compact and so is $\varphi^{-1}(K)\subset i(M)L$. If moreover, $Z=\t$, then, by \lemref{center}, $\varphi$ is one-to-one. So is $\hat\varphi$. Therefore $\varphi$ has also dense range. In conclusion, $\varphi$ is an isomorphism.
 \end{proof}

In this context, \thmref{main} becomes:
 
\begin{thm}\label{mainbis}  Let $(G,\Sigma)$ be a twisted  locally compact abelian group and let $S$ be a closed subgroup of $G$ such that $\Sigma_{|S}=p^{-1}(S)$ is abelian. Then $C^*(G,\Sigma)$ is isomorphic to $C^*(\hat S^\Sigma\rtimes H,\underline\Sigma)$ where $H=G/S$ acts affinely on $\hat S^\Sigma$ through the above map $\varphi_S$ and the twist $\underline\Sigma$ is obtained by the pushout construction.
 \end{thm}

\begin{proof} Note that we do not assume that $G$ is second countable because one can write the representations of $C^*(G,\Sigma)$ and $C^*(\hat S^\Sigma\rtimes H,\underline\Sigma)$ as integrated representations without this assumption. The subgroup $S$ satisfies the assumptions of \thmref{main}. Therefore $C^*(G,\Sigma)$ is isomorphic to $C^*(\hat S^\Sigma\rtimes H,\underline\Sigma)$. The action of $H$ on the twisted spectrum $\hat S^\Sigma$ is given by $(\chi h)(\sigma)=\chi(\tau\sigma\tau^{-1})$ where $\chi\in \hat S^\Sigma$, $\tau\in\Sigma$, $\sigma\in\Sigma_{|S}$ and $h=\pi_\Sigma(\tau)$. Since
\[\chi(\tau\sigma\tau^{-1})=\chi([\tau,\sigma]\sigma)=[\tau,\sigma]\chi(\sigma)=\hat h(\dot\sigma)\chi (\sigma),\]
this is indeed the affine action of $H$ through the map $\varphi_S$.
 \end{proof}

\begin{rem}
 Since the group $G$ is amenable, the C*-algebra $C^*(G,\Sigma)$ is nuclear and equal to the reduced C*-algebra $C_{\rm red}^*(G,\Sigma)$ (in \cite[Section II.3]{ren:approach}, this is proved when the twist is given by a 2-cocycle, but the same proof holds in the general case).
\end{rem}

For the choice $S=p(Z)$, where $Z$ is the center of $\Sigma$, \thmref{mainbis} provides the following reduction to the irreducible case.

\begin{thm}\cite[Proposition 34]{gre:local}
Let $(G,\Sigma)$ be a twisted locally compact abelian group. Let $S$ be the image in $G$ of the center $Z$ of $\Sigma$. The twisted C*-algebra $C^*(G,\Sigma)$ is isomorphic to the C*-algebra of a locally compact bundle of isomorphic irreducible twisted abelian groups $(H,\Sigma_\chi)_{\chi\in \hat S^\Sigma}$, where $H=G/S$ and $\hat S^\Sigma$ is the twisted spectrum of $S$. Therefore, it is the sectional algebra of a continuous field of C*-algebras of irreducible twisted abelian groups.
\end{thm}

\begin{proof} Since in this case, $\Sigma_{|S}=Z$ is the center of the group $\Sigma$, we have for $h=\pi_\Sigma(\tau)$ with $\tau\in\Sigma$ and $s=p(\sigma)$ with $\sigma\in Z$, $\varphi_S(h)(s)=[\tau,\sigma]=1$. Therefore, the action of $H=G/S$ on the twisted spectrum $\hat S^\Sigma$ is trivial. Thus the semi-direct product $\hat S^\Sigma\rtimes H$ is the cartesian product $\hat S^\Sigma\times H$. The twist $\underline \Sigma$ obtained by the pushout construction is of the form
\[\underline\Sigma=\{(\chi,a),\quad \chi\in \hat S^\Sigma, a\in\Sigma_\chi\}\]
where $\Sigma_\chi$ is the twist over $H$ obtained by the following pushout construction:
\[\begin{CD}
Z@>{}>{}>\Sigma @>{\pi}>{}>H\\
 @V{\chi}V{}V  @V{\chi_{*}}V{}V@|\\
\t @>{}>{}>\Sigma_\chi@>{}>{}> H
\end{CD}\]
According to \thmref{mainbis}, $C^*(G,\Sigma)$ is isomorphic to  $C^*(\hat S^\Sigma\times H, \underline\Sigma)$. The twisted groupoid $(\hat S^\Sigma\times H, \underline\Sigma)$ is a twisted group bundle whose bundle map is the projection onto $\hat S^\Sigma$ and fibers are the twisted abelian groups $(H,\Sigma_\chi)$. Let us show that they are  irreducible. Recall that an element $\chi\in \hat S^\Sigma$ corresponds to a trivialization of the twist 
$$\t\stackrel{j}{\rightarrowtail}Z\stackrel{p}{\twoheadrightarrow}  S.$$ 
Every $z\in Z$ can be written uniquely as
$z=\chi(z)\sigma_\chi(\dot z)$  where $\dot z=p(z)$ and $\sigma_\chi: S\to Z$ is a homomorphism such that $p\circ\sigma_\chi=id_S$. One checks that the pushout $\Sigma_\chi$ can be identified to the quotient group $\Sigma/\sigma_\chi(S)$. Indeed, along with the obvious maps, this quotient group satisfies the properties of the pushout:
\[\begin{CD}
Z@>{}>{}>\Sigma @>{\pi}>{}>H\\
 @V{\chi}V{}V  @V{\chi_{*}}V{}V@|\\
\t @>{}>{}>\Sigma/\sigma_\chi(S)@>{\pi'}>{}> H
\end{CD}\] 
Here the map $\chi_*$ is the quotient map. The relation $z=\chi(z)\sigma_\chi(\dot z)$ shows that  its restriction to $Z$ is $\chi$. By uniqueness of the pushout (see for example \cite[Theorem 1.5]{ikrsw:pushouts}), $\Sigma/\sigma_\chi(S)$ is  the pushout up to isomorphism. Let us show that this twist is irreducible. Let $a,b\in\Sigma$ such that their commutator $[a,b]$ belongs to $\sigma_\chi(S)$, say $[a,b]=\sigma_\chi(s)$. Since $[a,b]\in\t$, $p([a,b])=1$. Hence $s=p\circ\sigma_\chi(s)=1$ and $[a,b]=\sigma(1)=1$. If this true for all $b\in\Sigma$, $a$ belongs to the center $Z$ of $\Sigma$. Then $\chi_*(a)=\chi(a)$ belongs to $\t$. Let us show that the twisted groups $(H,\Sigma_\chi)_{\chi\in \hat S^\Sigma}$ are all isomorphic. Let $\chi,\chi'\in \hat S^\Sigma$. Since the group $\hat S$ acts transitively on $ \hat S^\Sigma$, there exists  $\rho\in\hat S$ such that $\chi'=\rho\chi$. Since the restriction map $\hat i:\hat G\to\hat S$ is surjective, there exists $\tilde\rho\in\hat G$ which extends $\rho$. Then, the map $\tau\mapsto \tilde\rho(\dot\tau)\tau$ is an automorphism of $\Sigma$ which induces an isomorphism $\rho_*: \Sigma/\sigma_\chi(S)\to \Sigma/\sigma_\chi'(S)$. Since it is the identity map on $\t$, it is a twist isomorphism. For the last statement, we have already noted that a bundle of amenable groups, possibly endowed with a twist, give rise to a continuous field of C*-algebras.
\end{proof}

\begin{rem}
Theorem 3.1 of \cite{bk:multiplier} says that every twist $(G,\Sigma)$ is the pullback of an irreducible twist $(G/S,\Sigma_1)$. One can check that $\Sigma$ is the pullback of $\Sigma_\chi$ with $\chi\in\hat S^\Sigma$.
\end{rem}

This theorem reduces the analysis of the C*-algebra of a twisted abelian group to the irreducible case. We consider now an irreducible twisted locally compact abelian group $(G,\Sigma)$ and choose a maximal abelian subgroup in $\Sigma$. It is necessarily of the form $\Sigma_{|S}$ where $S$ is a closed subgroup of $G$.

\begin{thm} Let $(G,\Sigma)$ be an irreducible twisted locally compact abelian group and let $S$ be a closed subgroup of $G$ such that $\Sigma_{|S}$ is maximal abelian in $\Sigma$. Then $C^*(G,\Sigma)$ is isomorphic to $C^*(\hat S^\Sigma\rtimes H,\underline\Sigma)$ where the action of $H=G/S$ on $\hat S^\Sigma$ is affine, free and minimal.   
\end{thm}

\begin{proof}
 Since in this case, according \lemref{map}, the map $\varphi_S: H\to\hat S$ is one-to-one and has dense range, the affine action of $H$ on $\hat S^\Sigma$ is free and minimal. 
\end{proof}

\begin{cor}\cite[Proposition 32]{gre:local}  Let $(G,\Sigma)$ be an irreducible twisted second countable locally compact abelian group. Then $C^*(G,\Sigma)$ is simple and nuclear. It is isomorphic to the C*-algebra of compact operators on $L^2(\hat S)$ if and only if $\varphi$ is an isomorphism of $G$ onto $\hat G$.
\end{cor}

\begin{proof} According to \cite[Corollary 4.6]{ren:ideal}, this C*-algebra is simple  and we have already noted that it is equal to the reduced C*-algebra and that it is nuclear. If $\varphi$ is an isomorphism, so is $\varphi_S$. Then  $\hat S^\Sigma\rtimes H$ is isomorphic to the trivial groupoid $\hat S\times\hat S$. One deduces that $C^*(\hat S^\Sigma\rtimes H,\underline\Sigma)$ is isomorphic to the C*-algebra of compact operators ${\mathcal K}(L^2(\hat S))$. If $\varphi$ is not an isomorphism, $\varphi_S$ is not an isomorphism. Therefore, the orbits of the action of $H$ on $\hat S^\Sigma$ are not locally closed. There exist non-trivial ergodic measures  \cite[Theorem 2.1]{ram:mackey-glimm}) and therefore factor representations which are non-type I \cite[Theorem 5.4]{hah:regular}.
\end{proof}

\begin{rem}
 The groupoid $\hat S^\Sigma\rtimes H$ is uniquely ergodic, in the sense that the only non-zero measures on $\hat S^\Sigma$ which are invariant under $H$ are those defined by a Haar measure of $\hat S$. As far as I know, the bijective correspondence between invariant measures and traces of the C*-algebra of a (twisted) principal locally compact groupoid with a Haar system has only been established in the \'etale case. Therefore, we cannot conclude directly that  $C^*(G,\Sigma)\simeq C^*(\hat S^\Sigma\rtimes H,\underline\Sigma)$ has a unique trace. This is shown in \cite[Proposition 32]{gre:local} by using the dual action. Let us repeat the proof given there, using groupoid C*-algebras as much as possible. One introduces the dual action of $\hat G$ on $C^*(G,\Sigma)$, defined by
 \[\alpha_\chi f(\tau)=\overline{\chi(\dot\tau)}f(\tau),\qquad \chi\in\hat G,\quad f\in C_c(G,\Sigma),\quad \tau\in\Sigma.\]
 According to \cite[Theorem II.5.7]{ren:approach}, the crossed product C*-algebra $C^*(\hat G, C^*(G,\Sigma))$ is isomorphic to the twisted groupoid C*-algebra $C^*(G\ltimes G,\underline\Sigma)$, where $G$ acts on itself by left translation on itself (the second $G$ is a space) and $\underline \Sigma=\Sigma\ltimes G$. Since $G\ltimes G$ is a trivial groupoid, one can see that $C^*(G\ltimes G,\underline\Sigma)$ is an elementary C*-algebra. Let $E$ be the canonical generalized expectation of the crossed product $C^*(\hat G, C^*(G,\Sigma))$ onto $C^*(G,\Sigma)$. If $\tau$ is a densely defined lower semi-continuous trace on $C^*(G,\Sigma)$ which is invariant under $\alpha$, one defines  the induced weight $\hat\tau=\tau\circ E$, also called the dual weight, on the crossed product. It is a densely defined lower semi-continuous trace. The invariance is in fact always satisfied here: we have for $x=\dot\sigma$, $\tau\in\Sigma$ and $f\in C_c(G,\Sigma)$
 \[\alpha_{\hat x}(f)(\tau)=\overline{[\sigma,\tau]}f(\tau)=f([\sigma,\tau]\tau)=f(\sigma\tau\sigma^{-1})=\chi(\sigma)^{-1}f\chi(\sigma) (\tau),\]
 where $\chi(\sigma)$ is the canonical unitary in the multiplier algebra of $C^*(G,\Sigma)$.
 This shows that $\alpha_{\hat x}$ is a generalized inner automorphism, hence $\tau\circ\alpha_{\hat x}=\tau$. One deduces that $\tau\circ\alpha_\chi=\tau$
 for all $\chi\in\hat G$ when the twist is irreducible. Since $C^*(\hat G, C^*(G,\Sigma))$ has a unique trace and the map $\tau\mapsto\hat \tau$ is one-to-one, $C^*(G,\Sigma)$ has a unique trace.
\end{rem}

\section*{Acknowledgements}
I thank M.~Ionescu, A.~Kumjian, A.~Sims and D.~ Williams for a stimulating and enjoyable collaboration and L.~Baggett for illuminating comments about multipliers on abelian groups.

\end{document}